\documentclass[a4paper,12pt]{amsart}
\usepackage{amssymb,amscd,amsmath,latexsym,color}
\usepackage[dvips]{graphicx}

\calclayout

\theoremstyle{plain}

\newtheorem{theo}{Theorem}[section]
\newtheorem{prop}[theo]{Proposition}
\newtheorem{lemm}[theo]{Lemma}

\theoremstyle{definition}
\newtheorem*{defi}{Definition}
\newtheorem{exam}[theo]{Example}

\theoremstyle{remark}
\newtheorem*{rema}{Remark}

\numberwithin{equation}{section}

\newcommand{\field}[1]{\mathbb{#1}}
\newcommand{\C}{\field{C}}
\newcommand{\R}{\field{R}}
\newcommand{\Z}{\field{Z}}

\DeclareMathOperator{\sgn}{sgn}

\def\e{\epsilon}
\def\P{\mathcal P}
\def\Q{\mathcal Q}

\def\A{\mathcal A}
\def\ep{\epsilon}

\begin{document}

\title
[Lattice multi-polygons]{Lattice multi-polygons}

\author[A. Higashitani]{Akihiro Higashitani}
\address{Department of Mathematics, Kyoto Sangyo University, Motoyama, Kamigamo, Kita-Ku, Kyoto, 603-8555, Japan}
\email{ahigashi@cc.kyoto-su.ac.jp}
\author[M. Masuda]{Mikiya Masuda}
\address{Department of Mathematics, Graduate School of Science, Osaka City University, Sugimoto, 
Sumiyoshi-ku, Osaka 558-8585, Japan}
\email{masuda@sci.osaka-cu.ac.jp}
\thanks{{\bf 2010 Mathematics Subject Classification:} Primary 05A99, Secondary 51E12, 57R91 \\
\;\;\;\; {\bf Keywords:} Lattice polygon, twelve-point theorem, Pick's formula, Ehrhart polynomial, toric topology \\
\;\;\;\; The first author is supported by JSPS Research Fellowship for Young Scientists. 
The second author is partially supported by Grant-in-Aid for Scientific Research 22540094}

\maketitle

\begin{abstract}
We discuss generalizations of some results on lattice polygons to certain piecewise linear loops which may have a self-intersection but have vertices in the lattice $\Z^2$.  We first prove a formula on the rotation number of a unimodular sequence in $\Z^2$.  This formula implies the generalized twelve-point theorem in \cite{po-rv00}.  We then introduce the notion of lattice multi-polygons which is a generalization of lattice polygons, 
state the generalized Pick's formula and discuss the classification of Ehrhart polynomials of lattice multi-polygons and also of several natural subfamilies of lattice multi-polygons.   
\end{abstract}

\section*{Introduction}

Lattice polygons are an elementary but fascinating object.  Many interesting results such as Pick's formula are known for them.  However, not only the results are interesting, but also there are a variety of proofs to the results and some of them use advanced mathematics such as toric geometry, complex analysis and modular form (see \cite{fult93, diro95, oda88, po-rv00} for example).  These proofs are unexpected and make the study of lattice polygons more fruitful and intriguing.     

Some of the results on lattice polygons are generalized to certain generalized polygons.  For instance, Pick's formula \cite{pick99}
 \[A(P)=\sharp P^\circ + \frac{1}{2}B(P) - 1\] 
for a lattice polygon $P$, where $A(P)$ is the area of $P$ and $\sharp P^\circ$ (resp. $B(P)$) is the number of lattice points in the interior (resp. on the boundary) of $P$, is generalized in several directions and one of the generalizations is to certain piecewise linear loops which may have a self-intersection but have vertices in $\Z^2$ (\cite{Grum-She-93, masu99}).  
As is well known, Pick's formula has an interpretation in toric geometry when $P$ is convex (\cite{fult93, oda88}) 
but the proof {using toric geometry is} not applicable when $P$ is concave.  However, once we develop toric geometry from the topological point of view, that is \emph{toric topology}, Pick's formula can be proved along the same line in full generality as is done in \cite{masu99}.  

Another such result on lattice polygons is the twelve-point theorem.  It says that if $P$ is a convex lattice polygon which contains the origin in its interior as a unique lattice point, then 
\[B(P)+B(P^\vee)=12,\] 
where $P^\vee$ is the lattice polygon dual to $P$. 
Several proofs are known to the theorem and one of them again uses toric geometry.  B. Poonen and F. Rodriguez-Villegas \cite{po-rv00} provided a new proof using modular forms. They also formulate a generalization of the twelve-point theorem and claim that their proof works in the general setting.  
It is mentioned in \cite{po-rv00} that the proof using toric geometry is difficult to generalize, but a slight generalization of the proof of \cite[Theorem 5.1]{masu99}, which uses toric topology and is on the same line of the proof using toric geometry, implies the generalized twelve-point theorem.  

{Generalized polygons considered in the generalization of the twelve-point theorem are what is called \emph{legal loops}.  A legal loop may have a self-intersection and is associated to a unimodular sequence of vectors $v_1,\dots,v_d$ in $\Z^2$.  Here unimodular means that any consecutive two vectors $v_i,v_{i+1}$ $(i=1,\dots,d)$ in the sequence form a basis of $\Z^2$, where $v_{d+1}=v_1$.  Therefore, $\e_i=\det(v_i,v_{i+1})$ is $\pm 1$.  One sees that there is a unique integer $a_i$ satisfying 
\[
\e_{i-1}v_{i-1}+\e_i v_{i+1}+a_iv_i=0
\]
for each $i=1,\dots,d$.  Note that $|a_i|$ is twice the area of the triangle with vertices $v_{i-1}, v_{i+1}$ and the origin.  We prove that the rotation number of the unimodular sequence $v_1,\dots,v_d$ around the origin is given by 
\[
\frac{1}{12}\big(\sum_{i=1}^da_i+3\sum_{i=1}^d\e_i\big)
\]
(see Theorem~\ref{rotation}).  The generalized twelve-point theorem easily follows from this formula.  This formula was originally proved using toric topology which requires some advanced topology, but after that, an elementary and combinatorial proof was found.  We give it in Section~\ref{sect:1} and the original proof in the Appendix.  A different elementary proof to the above formula appeared in \cite{ziv12} while revising this paper.} 

We also introduce the notion of lattice multi-polygons.  A \emph{lattice multi-polygon} is a piecewise linear loop with vertices in $\Z^2$ together with a sign function which assigns either $+$ or $-$ to each side and satisfies some mild condition. The piecewise linear loop may have a self-intersection and we think of it as a sequence of points in $\Z^2$.  A lattice polygon can naturally be regarded as a lattice multi-polygon.  The generalized Pick's formula holds for lattice multi-polygons, so Ehrhart polynomials can be defined for them.  The Ehrhart polynomial of a lattice multi-polygon is of degree at most two. The constant term is the rotation number of normal vectors to sides of the multi-polygon and not necessarily $1$ unlike ordinary Ehrhart polynomials.  The other coefficients have similar geometrical meaning to the ordinary ones but they can be zero or negative unlike the ordinary ones.  
The family of lattice multi-polygons has some natural subfamilies, e.g. the family of all convex lattice polygons. We discuss the characterization of Ehrhart polynomials of not only all lattice multi-polygons but also some natural subfamilies. 

The structure of the present paper is as follows.  In Section 1, we {give the elementary proof to the formula} which describes the rotation number of a unimodular sequence of vectors in $\Z^2$ around the origin. Here the vectors in the sequence may go back and forth. 
The proof using toric topology is given in the Appendix. 
In Section 2, we observe that {the formula} implies the generalized twelve-point theorem.  In Section 3, we introduce the notion of lattice multi-polygon and state the generalized Pick's formula for lattice multi-polygons.  In Section 4, we discuss the characterization of Ehrhart polynomials of lattice multi-polygons and of several natural subfamilies of lattice multi-polygons.

\section{Rotation number of a unimodular sequence} \label{sect:1}

We say that a sequence of vectors $v_1,\dots,v_d$ in $\Z^2$ $(d\ge 2)$ is \emph{unimodular} 
if each triangle with vertices ${\bf 0}, v_i$ and $v_{i+1}$ contains no lattice point except the vertices,
where ${\bf 0}=(0,0)$ and $v_{d+1}=v_1$. 
The vectors in the sequence are not necessarily counterclockwise or clockwise. 
They may go back and forth.  We set 
\begin{equation} \label{ei}
\e_i=\det(v_i,v_{i+1}) \quad\text{for $i=1,\dots,d$}.
\end{equation}
In other words, $\e_i=1$ if the rotation from $v_i$ to $v_{i+1}$ (with angle less than $\pi$) is counterclockwise and $\e_i=-1$ otherwise. 
Since each successive pair $(v_{j},v_{j+1})$ is a basis of $\Z^2$ for $j=1,\dots,d$, one has 
\begin{equation*} \label{matrix}
(v_i,v_{i+1})=(v_{i-1},v_i)\begin{pmatrix} 0 &-\e_{i-1}\e_i\\
1 & -\e_ia_i\end{pmatrix}
\end{equation*}
with a unique integer $a_i$ for each $i$.  This is equivalent to
\begin{equation} \label{ai}
\e_{i-1}v_{i-1}+\e_iv_{i+1}+a_iv_i=0.
\end{equation}
{Note that $|a_i|$ is twice the area of the triangle with vertices ${\bf 0}, v_{i-1}$ and $v_{i+1}$.}

\begin{exam}\label{ex1}
(a) {Take a unimodular sequence} 
 \[\P=(v_1,\ldots,v_5)=((1,0),(0,1),(-1,0),(0,-1),(-1,-1)),\] 
{see Figure~\ref{legal1} in Section~\ref{sect:2}.}
Then 
\[\e_1=\e_2=\e_3=\e_5=1,\ \e_4=-1\quad \text{and} \quad a_1=a_4=a_5=1,\ a_2=a_3=0\] 
and the rotation number of $\P$ {around the origin} is 1. 

(b) {Take another unimodular sequence}
\[\Q=(v_1,\ldots,v_6)=((1,0),(-1,1),(0,-1),(1,1),(-1,0),(1,-1)),\] 
{see Figure~\ref{legal2} in Section~\ref{sect:2}.}
Then 
\[\e_1=\cdots=\e_6=1 \;\;\text{and}\;\; a_1=a_6=0,\ a_2=a_4=1,\ a_3=a_5=2\] 
and the rotation number {of $\Q$ around the origin} is 2.
\end{exam}

Our main result in this section is the following. 

\begin{theo} \label{rotation}
The rotation number of a unimodular sequence $v_1,\dots,v_d$ $(d\ge 2)$ around the origin is given by 
\begin{equation}\label{12}
\frac{1}{12}\big(\sum_{i=1}^da_i+3\sum_{i=1}^d\e_i\big)
\end{equation}
{where $\e_i$ and $a_i$ are the integers defined in \eqref{ei} and \eqref{ai}.}
\end{theo}

For our proof of this theorem, we prepare the following lemma.  
\begin{lemm}\label{norm}
Let $v_1,\ldots,v_d$ be a unimodular sequence and $v_j$ a vector 
whose Euclidean norm is maximal among the vectors in the sequence, where $1 \leq j \leq d$. 
Then $a_j=0$ or $\pm 1$.
\end{lemm}
\begin{proof}
{It follows from \eqref{ai} and the maximality of the Euclidean norm of $v_j$ that we have} 
\begin{equation} \label{trieq}
 \|a_jv_j\|=\|-\e_{j-1}v_{j-1}-\e_jv_{j+1}\| \leq \|v_{j-1}\|+\|v_{j+1}\| \leq \|v_j\|+\|v_j\|, 
\end{equation}
where $\|\ \|$ denotes the Euclidean norm on $\R^2$.  
{Therefore, $|a_j|\le 1$ or $|a_j|=2$ and the equality holds in \eqref{trieq}.  However, the latter case does not occur because the vectors $v_{j-1}, v_j, v_{j+1}$ are not parallel, proving the lemma.} 
\end{proof}

\begin{proof}[Proof of Theorem \ref{rotation}]
We give a proof by induction on $d$.  

When $d=2$, the rotation number of $v_1,v_2$ is zero while $a_1=a_2=0$ and $\e_1+\e_2=0$.  Therefore the theorem holds in this case.  

When $d=3$, we may assume that $(v_1,v_2)=((1,0),(0,1))$ or $(v_1,v_2)=((0,1),(1,0))$ through an {(orientation preserving)} unimodular transformation on $\R^2$, 
and then $v_3$ is one of $(1,1),(-1,1),(1,-1)$ and $(-1,-1)$.
Now, it is immediate to check that 
the rotation number of each unimodular sequence coincides with \eqref{12}. 

Let $d \geq 4$ and assume that the theorem holds for any unimodular sequence with at most $d-1$ vectors. 
Let $v_j$ be a vector in the unimodular sequence $v_1,\ldots,v_d$ whose Euclidean norm is maximal among the vectors in the sequence. 
Then Lemma \ref{norm} says that $a_j=0$ or $\pm 1$. \\

\underbar{The case where $a_j=0$}, i.e. 
\begin{equation} \label{r-0}
\e_{j-1}v_{j-1}+\e_jv_{j+1}=0.
\end{equation}
In this case, we consider a subsequence $v_1,\ldots,v_{j-2},v_{j+1},\ldots,v_d$ obtained by removing two vectors $v_{j-1}$ and $v_j$ from the given unimodular sequence.  
Since \[|\det(v_{j-2},v_{j+1})|=|\det(v_{j-2},-\e_{j-1}\e_jv_{j-1})|=1,\] the subsequence is also unimodular. 
Set 
\begin{equation} \label{r-1}
v_i'=\begin{cases} v_i \quad&\text{for $1\le i\le j-2$},\\
v_{i+2} \quad&\text{for $j-1\le i\le d-2$}
\end{cases}
\end{equation}
and define $\e_i'$ and $a_i'$ for the unimodular sequence $v_1',\ldots,v_{d-2}'$ similarly to \eqref{ei} and \eqref{ai}, i.e., 
\begin{equation} \label{r-2}
\e_i'=\det(v_i',v_{i+1}'),\quad \e_{i-1}'v_{i-1}'+\e_i'v_{i+1}'+a_i'v_i'=0.
\end{equation}
Then, it follows from \eqref{r-0}, \eqref{r-1}, \eqref{r-2} and \eqref{ei} that   
\begin{equation} \label{r-3}
\e_i'=\begin{cases} \e_i \quad&\text{for $1\le i\le j-3$},\\
-\e_{j-2}\e_{j-1}\e_j \quad&\text{for $i=j-2$},\\
\e_{i+2} \quad&\text{for $j-1\le i\le d-2$}.
\end{cases}
\end{equation}
It also follows from \eqref{r-0}, \eqref{r-1}, \eqref{r-2}, \eqref{r-3} and \eqref{ai} that 
\begin{align*}
a_{j-2}'v_{j-2}=a_{j-2}'v_{j-2}'&=-\e_{j-3}'v_{j-3}'-\e_{j-2}'v_{j-1}' \\
&=-\e_{j-3}v_{j-3}-(-\e_{j-2}\e_{j-1}\e_j)(-\e_{j-1}\e_jv_{j-1}) \\
&=-\e_{j-3}v_{j-3}-\e_{j-2}v_{j-1}=a_{j-2}v_{j-2}
\end{align*} and \begin{align*}
a_{j-1}'v_{j+1}=a_{j-1}'v_{j-1}'&=-\e_{j-2}'v_{j-2}'-\e_{j-1}'v_j' \\
&=\e_{j-2}\e_{j-1}\e_jv_{j-2}-\e_{j+1}v_{j+2} \\
&=-\e_{j-1}\e_j(-\e_{j-2}v_{j-2}-\e_{j-1}v_j)-\e_jv_j-\e_{j+1}v_{j+2} \\
&=-\e_{j-1}\e_ja_{j-1}v_{j-1}+a_{j+1}v_{j+1}\\
&=a_{j-1}v_{j+1}+a_{j+1}v_{j+1}=(a_{j-1}+a_{j+1})v_{j+1}. 
\end{align*}
Therefore
\begin{equation} \label{r-4}
a_i'=\begin{cases} a_i\quad&\text{for $1\le i\le j-2$},\\
a_{j-1}+a_{j+1} \quad&\text{for $i=j-1$},\\
a_{i+2}\quad&\text{for $j\le i\le d-2$}.
\end{cases}
\end{equation}

Since $a_j=0$, it follows from \eqref{r-3} and \eqref{r-4} that 
\begin{equation} \label{r-5} 
\begin{split}
&\frac{1}{12}\big(\sum_{i=1}^da_i+3\sum_{i=1}^d\e_i\big) -\frac{1}{12}\big(\sum_{i=1}^{d-2}a_i'+3\sum_{i=1}^{d-2}\e_i'\big)\\
=& \frac{1}{4}(\e_{j-2}+\e_{j-1}+\e_j-\e_{j-2}')
= \frac{1}{4}(\e_{j-2}+\e_{j-1}+\e_j+\e_{j-2}\e_{j-1}\e_j)
\end{split}
\end{equation}
which is $+1$ (resp. $-1$) if $\e_{j-2}, \e_{j-1}$ and $\e_j$ are all $+1$ (resp. $-1$), and $0$ otherwise.  
On the other hand, one can see that if the rotation number of $v_1,\ldots,v_d$ is $r$, then that of $v_1',\ldots,v_{d-2}'$ is equal to $r-1$ (resp. $r+1$) if $\e_{j-2}, \e_{j-1}$ and $\e_j$ are all $+1$ (resp. $-1$), and $r$ otherwise.  This together with \eqref{r-5} and the the hypothesis of induction shows that $\frac{1}{12}\big(\sum_{i=1}^da_i+3\sum_{i=1}^d\e_i\big)$ is the rotation number of $v_1,\ldots,v_d$.  \\
 
\underbar{The case where $a_j=\pm 1$}.  
We have 
\begin{equation} \label{r-6}
\e_jv_{j+1}+\e_{j-1}v_{j-1}+a_j v_j=0.
\end{equation}
In this case, we consider a subsequence $v_1,\ldots,v_{j-1},v_{j+1},\ldots,v_d$ obtained by removing the $v_j$ from the given unimodular sequence.  
Since 
\[
|\det(v_{j-1},v_{j+1})|=|\det(v_{j-1},-\e_j\e_{j-1}v_{j-1}-\e_ja_j v_j)|=|\det(v_{j-1},v_j)|=1,
\] 
the subsequence is also unimodular. 
Set 
\begin{equation} \label{r-7}
v_i'=\begin{cases} v_i \quad&\text{for $1\le i\le j-1$},\\
v_{i+1} \quad&\text{for $j\le i\le d-1$}
\end{cases}
\end{equation}
and define $\e_i'$ and $a_i'$ for the unimodular sequence $v_1',\ldots,v_{d-1}'$ as before by \eqref{r-2}. 
Then, it follows from \eqref{r-2}, \eqref{r-6}, \eqref{r-7} and \eqref{ei} that   
\begin{equation} \label{r-8}
\e_i'=\begin{cases} \e_i \quad&\text{for $1\le i\le j-2$},\\
-\e_{j-1}\e_ja_j \quad&\text{for $i=j-1$},\\
\e_{i+1} \quad&\text{for $j\le i\le d-1$}.
\end{cases}
\end{equation}
It also follows from \eqref{r-6}, \eqref{r-7}, \eqref{r-8}, \eqref{r-2} and \eqref{ai} that  
\begin{align*}
a_{j-1}'v_{j-1}=a_{j-1}'v_{j-1}'&=-\e_{j-2}'v_{j-2}'-\e_{j-1}'v_j' \\
&=-\e_{j-2}v_{j-2}+\e_{j-1}\e_ja_j v_{j+1} \\
&=-\e_{j-2}v_{j-2}+\e_{j-1}a_j(-a_j v_j-\e_{j-1}v_{j-1})\\
&=-\e_{j-2}v_{j-2}-\e_{j-1}v_j-a_j v_{j-1}\\
&=a_{j-1}v_{j-1}-a_j v_{j-1}=(a_{j-1}-a_j)v_{j-1}
\end{align*} and \begin{align*}
a_j'v_{j+1}=a_j'v_j'&=-\e_{j-1}'v_{j-1}'-\e_j'v_{j+1}' \\
&=\e_{j-1}\e_ja_j v_{j-1}-\e_{j+1}v_{j+2} \\
&=\e_ja_j (-a_j v_j-\e_jv_{j+1})-\e_{j+1}v_{j+2}\\
&=-a_j v_{j+1}-\e_jv_j-\e_{j+1}v_{j+2}\\
&=-a_j v_{j+1}+a_{j+1}v_{j+1}=(a_{j+1}-a_j)v_{j+1}.  
\end{align*}
Therefore 
\begin{equation} \label{r-9}
a_i'=\begin{cases} a_i\quad&\text{for $1\le i\le j-2$},\\
a_{j-1}-a_j \quad&\text{for $i=j-1$},\\
a_{j+1}-a_j \quad&\text{for $i=j$},\\
a_{i+1}\quad&\text{for $j+1\le i\le d-1$}.
\end{cases}
\end{equation}

It follows from \eqref{r-8} and \eqref{r-9} that 
\begin{equation} \label{r-10} 
\begin{split}
&\frac{1}{12}\big(\sum_{i=1}^da_i+3\sum_{i=1}^d\e_i\big) -\frac{1}{12}\big(\sum_{i=1}^{d-1}a_i'+3\sum_{i=1}^{d-1}\e_i'\big)\\
=& \frac{1}{4}\big(a_j+\e_{j-1}+\e_j-\e_{j-1}'\big)
=\frac{1}{4}\big((1+\e_{j-1}\e_j)a_j+\e_{j-1}+\e_j\big)
\end{split}
\end{equation}
which is $a_j$ if both $\e_{j-1}$ and $\e_j$ are $a_j$, and $0$ otherwise.  
On the other hand, one can see that if the rotation number of $v_1,\ldots,v_d$ is $r$, then that of $v_1',\ldots,v_{d-1}'$ is equal to $r-a_j$ if both $\e_{j-1}$ and $\e_j$ are $a_j$, and $r$ otherwise.  This together with \eqref{r-10} and the the hypothesis of induction shows that $\frac{1}{12}\big(\sum_{i=1}^da_i+3\sum_{i=1}^d\e_i\big)$ is the rotation number of $v_1,\ldots,v_d$.  

This completes the proof of the theorem. 
\end{proof}

\begin{rema}
A different elementary proof to Theorem~\ref{rotation} is given in \cite{ziv12}.
\end{rema}

\section{Generalized twelve-point theorem} \label{sect:2}

Let $P$ be a convex lattice polygon whose only interior lattice point is the origin.  
Then the dual $P^\vee$ to $P$ is also a convex lattice polygon whose only interior lattice point is the origin.  
Let $B(P)$ denote the total number of the lattice points on the boundary of $P$.  The following fact is well known. 

\begin{theo}[Twelve-point theorem]
$B(P)+B(P^\vee)=12$.
\end{theo}

Several proofs are known for this theorem (\cite{cast12, ce-re-sk05, po-rv00}).  B. Poonen and F. Rodriguez-Villegas give a proof using modular forms in \cite{po-rv00}. They also formulate a generalization of the twelve-point theorem and claim that their proof works in the general setting.  
In this section, we will explain the generalized twelve-point theorem and observe that it follows from Theorem~\ref{rotation}.      

If $P$ is a convex lattice polygon whose only interior lattice point 
is the origin and $v_1,\dots,v_d$ are the vertices of $P$ arranged counterclockwise, 
then every $v_i$ is primitive and the triangle with the vertices ${\bf 0}, v_i$ and $v_{i+1}$ has no lattice point in the interior for each $i$, where $v_{d+1}=v_1$ as usual. 
This observation motivates the following definition, see \cite{po-rv00, cast12}. 

\begin{defi}
A sequence of vectors $\P=(v_1,\dots,v_d)$, where $v_1,\dots,v_d$ are in $\Z^2$ and $d\ge 2$, is called a \emph{legal loop} if every $v_i$ is primitive and whenever $v_i\not=v_{i+1}$, $v_i$ and $v_{i+1}$ are linearly independent (i.e. $v_i\not=-v_{i+1}$) and the triangle with the vertices ${\bf 0}, v_i$ and $v_{i+1}$ has no lattice point in the interior.  We say that a legal loop is \emph{reduced} if $v_i\not=v_{i+1}$ for any $i$.  
A (non-reduced) legal loop $\P$ naturally determines a reduced legal loop, denoted $\P_{red}$, by dropping all the redundant points.   We define the \emph{winding number} of a legal loop $\P=(v_1,\dots,v_d)$ to be the rotation number of the vectors $v_1,\dots,v_d$ around the origin.  
\end{defi} 

Joining successive points in a legal loop $\P=(v_1,\dots,v_d)$ by straight lines forms a lattice polygon which may have a self-intersection.  
A unimodular sequence $v_1,\dots,v_d$ determines a reduced legal loop.  
Conversely, a reduced legal loop $\P=(v_1,\dots,v_d)$ determines a unimodular sequence by adding all the lattice points on the line segment $v_iv_{i+1}$ (called a \emph{side} of $\P$)  connecting $v_i$ and $v_{i+1}$ for every $i$.  
To each side $v_iv_{i+1}$ with $v_i\not=v_{i+1}$, we assign the sign of $\det(v_i,v_{i+1})$, denoted $\sgn(v_i,v_{i+1})$.     

For a reduced legal loop $\P=(v_1,\dots,v_d)$, 
we set  
\begin{equation} \label{wi}
w_i=\frac{v_i-v_{i-1}}{\det(v_{i-1},v_i)} \quad\text{for $i=1,\dots,d$}, 
\end{equation}  
where $v_0=v_d$. 
{Note that $w_i$ is integral and primitive and define $\P^\vee=(w_1,\dots,w_d)$ following \cite{po-rv00} (see also \cite{cast12}).} 
It is not difficult to see that $\P^\vee=(w_1,\dots,w_d)$ is again a legal loop 
although it may not be reduced (see the proof of Theorem~\ref{gen12} below).  
If a legal loop $\P$ is not reduced, then we define $\P^\vee$ to be $(\P_{red})^\vee$.  
When the vectors $v_1,\dots,v_d$ are the vertices of a convex lattice polygon $P$ with only the origin as an interior lattice point and are arranged in counterclockwise order, 
the sequence $w_1,\dots,w_d$ is also in counterclockwise order and the convex hull of $w_1,\dots,w_d$ is the 180 degree rotation of the polygon  $P^\vee$ dual to $P$.  

\begin{exam}\label{ex2}
Let us consider $\P$ and $\Q$ described in Example \ref{ex1}. Then those are reduced legal loops. \\
(a) We have \[\P^\vee=((2,1),(-1,1),(-1,-1),(1,-1),(1,0)).\] 

\begin{figure}[htb!]
\centering
\includegraphics[scale=0.72]{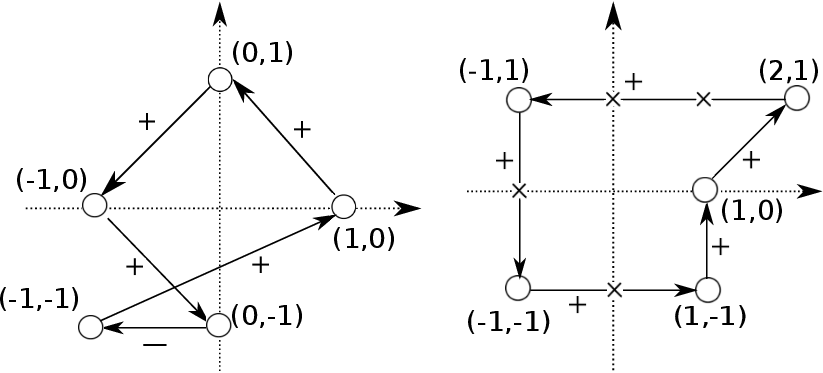}
\caption{legal loops $\P$ and $\P^\vee$ and sides with signs}\label{legal1}
\end{figure}

(b) Similarly, \[\Q^\vee=((0,1),(-2,1),(1,-2),(1,2),(-2,-1),(2,-1)).\] 

\begin{figure}[htb!]
\centering
\includegraphics[scale=0.8]{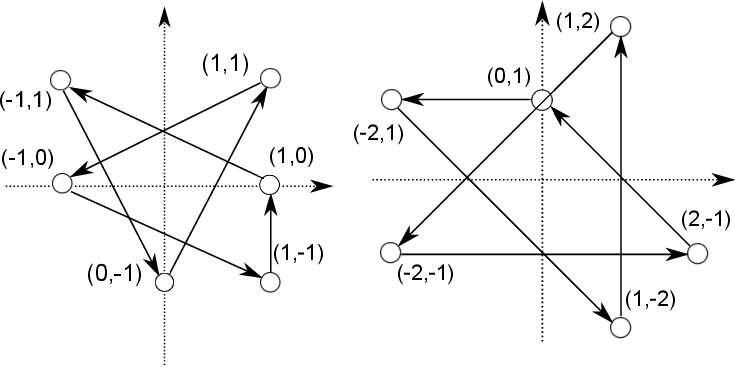}
\caption{leagl loops $\Q$ and $\Q^\vee$}\label{legal2}
\end{figure}
\end{exam}

\begin{defi}
Let $|v_iv_{i+1}|$ be the number of lattice points on the side $v_iv_{i+1}$ minus $1$, so $|v_iv_{i+1}|=0$ when $v_i=v_{i+1}$.  Then we define 
\[
B(\P)=\sum_{i=1}^d\sgn(v_i,v_{i+1})|v_iv_{i+1}|.
\]
Clearly, $B(\P)=B(\P_{red})$. 
\end{defi}

\begin{theo}[Generalized twelve-point theorem {\cite{po-rv00}}]\label{gen12}
Let $\P$ be a legal loop and let $r$ be the winding number of $\P$. Then 
$B(\P)+B(\P^\vee)=12 r$. 
\end{theo}
\begin{proof}
We may assume that $\P$ is reduced.  As remarked before, the reduced legal loop $\P=(v_1,\dots,v_d)$ determines a unimodular sequence by adding all the lattice points on the side $v_iv_{i+1}$ for every $i$, and the unimodular sequence determines a reduced legal loop, say $\mathcal Q$.  Clearly, $B(\P)=B(\mathcal Q)$ and $(\P^\vee)_{red}=(\mathcal Q^\vee)_{red}$.  
In the sequel, we may assume that the vectors $v_1,\dots,v_d$ in our legal loop $\P$ form a unimodular sequence.  

Since the sequence $v_1,\dots,v_d$ is unimodular, $\sgn(v_i,v_{i+1})=\e_i$ and $|v_iv_{i+1}|=1$ for any $i$.  Therefore 
\begin{equation} \label{lP}
B(\P)=\sum_{i=1}^d\sgn(v_i,v_{i+1})|v_iv_{i+1}|=\sum_{i=1}^d\e_i.
\end{equation}

On the other hand, it follows from \eqref{wi} and \eqref{ai} that 
\begin{equation} \label{wminus}
\begin{split}
w_{i+1}-w_i&=\e_i(v_{i+1}-v_i)-\e_{i-1}(v_i-v_{i-1})\\
&=\e_iv_{i+1}+\e_{i-1}v_{i-1}-(\e_i+\e_{i-1})v_i\\
&=-(a_i+\e_i+\e_{i-1})v_i
\end{split}
\end{equation}
and that 
\begin{equation} \label{wdet}
\begin{split}
\det(w_i,w_{i+1})&=\e_{i-1}\e_i\det(v_i-v_{i-1},v_{i+1}-v_i)\\
&=\e_{i-1}\e_i\det(v_i-v_{i-1},-\e_{i-1}\e_iv_{i-1}-\e_ia_iv_i-v_i)\\
&=\e_{i-1}\e_i\big(\det(v_i,-\e_{i-1}\e_iv_{i-1})+\det(-v_{i-1},-\e_ia_iv_i-v_i)\big)\\
&=\e_{i-1}+a_i+\e_i.
\end{split}
\end{equation}
Since $v_i$ is primitive, \eqref{wminus} shows that $|w_iw_{i+1}|=|\e_{i-1}+\e_i+a_i|$ and this together with \eqref{wdet} shows that 
\[
\sgn(w_i,w_{i+1})|w_iw_{i+1}|=\e_{i-1}+\e_i+a_i.
\]
Therefore  
\begin{equation} \label{lPv}
B(\P^\vee)=\sum_{i=1}^d\sgn(w_i,w_{i+1})|w_iw_{i+1}|=\sum_{i=1}^d(\e_{i-1}+\e_i+a_i).
\end{equation}

It follows from \eqref{lP} and \eqref{lPv} that 
\[
\begin{split}
B(\P)+B(\P^\vee)&=\sum_{i=1}^d\e_i+\sum_{i=1}^d(\e_{i-1}+\e_i+a_i)\\
&=3\sum_{i=1}^d\e_i+\sum_{i=1}^da_i,
\end{split}
\]
which is equal to $12r$ by Theorem~\ref{rotation}, proving the theorem. 
\end{proof}

\begin{exam}
Let us consider again the legal loops $\P$ and $\Q$ in the previous example. \\
(a) On the one hand, $B(\P)=1+1+1-1+1=3$. On the other hand, $B(\P^\vee)=3+2+2+1+1=9$. 
Thus we have $B(\P)+B(\P^\vee)=12$. 
The left-hand side (resp. right-hand side) of Figure \ref{legal1} depicted in Example \ref{ex2} 
shows $\P$ (resp. $\P^\vee$) together with signs, 
where the symbols $\circ$ and $\times$ stand for lattice points in $\Z^2$. \\
(b) On the one hand, $B(\Q)=6$. On the other hand, $B(\Q^\vee)=18$. Hence, $B(\Q)+B(\Q^\vee)=24$. 
The left-hand side (resp. right-hand side) of Figure \ref{legal2} 
shows $\Q$ (resp. $\Q^\vee$). Note that the signs on the sides of $\Q$ and $\Q^\vee$ are all $+$. 

\end{exam}

\begin{rema}
Kasprzyk and Nill (\cite[Corollary 2.7]{ka-ni11}) point out that the generalized twelve-point theorem can further be generalized to what are called \emph{$\ell$-reflexive loops}, where $\ell$ is a positive integer and a $1$-reflexive loop is a {legal loop}. 
\end{rema}

\section{Generalized Pick's formula for lattice multi-polygons}\label{Ehrhart}

In this section, we introduce the notion of lattice multi-polygon and 
state a generalized Pick's formula for lattice multi-polygons which is essentially proved in \cite[Theorem 8.1]{masu99}.  
Moreover, from this formula, we can define the Ehrhart polynomials for lattice multi-polygons. 

We begin with the well-known Pick's formula for lattice polygons (\cite{pick99}).  Let $P$ be a (not necessarily convex) lattice polygon, $\partial P$ the boundary of $P$ and $P^\circ=P\backslash \partial P$.  We define
\[
A(P)=\text{the area of $P$},\quad B(P)=|\partial P\cap \Z^2|,\quad \sharp P^\circ=|P^\circ\cap \Z^2|,
\]
where $|X|$ denotes the cardinality of a finite set $X$.  Then Pick's formula says that 
\begin{equation} \label{eq:3.1}
A(P)=\sharp P^\circ+\frac{1}{2}B(P)-1.  
\end{equation}
We may rewrite \eqref{eq:3.1} as 
\[
\sharp P^\circ=A(P)-\frac{1}{2}B(P)+1\quad\text{or}\quad \sharp P=A(P)+\frac{1}{2}B(P)+1, 
\]
where $\sharp P=|P\cap\Z^2|$.  

In \cite{Grum-She-93}, the notion of \emph{shaven lattice polygon} is introduced and Pick's formula \eqref{eq:3.1} is generalized to shaven lattice polygons.  The generalization of Pick's formula discussed in \cite{masu99} is similar to \cite{Grum-She-93} but a bit more general, which we shall explain.  

Let $\P=(v_1,\dots,v_d)$ be a sequence of points $v_1,\dots,v_d$ in $\Z^2$. 
One may regard $\P$ as an \emph{oriented} piecewise linear loop by connecting all successive points from $v_i$ to $v_{i+1}$ in $\P$ by straight lines as before, where $v_{d+1}=v_1$. 
To each side $v_iv_{i+1}$, we assign a sign $+$ or $-$, denoted $\ep(v_iv_{i+1})$.  In Section~\ref{sect:2}, we assigned the $\sgn(v_i,v_{i+1})$, which is the sign of $\det(v_i,v_{i+1})$, to $v_iv_{i+1}$ but $\ep(v_iv_{i+1})$ may be different from $\sgn(v_i,v_{i+1})$.  However we require that the assignment $\ep$ of signs satisfy the following condition $(\star)$: 
\begin{itemize}
\item[$(\star)$] when there are consecutive three points $v_{i-1},v_i,v_{i+1}$ in $\P$ lying on a line, we have 
\begin{enumerate}
\item $\ep(v_{i-1}v_i)=\ep(v_iv_{i+1})$ if $v_i$ is in between $v_{i-1}$ and $v_{i+1}$; 
\item $\ep(v_{i-1}v_i) \not= \ep(v_iv_{i+1})$ if $v_{i-1}$ lies on $v_iv_{i+1}$ or $v_{i+1}$ lies on $v_{i-1}v_i$. 
\end{enumerate}
\end{itemize}

A \emph{lattice multi-polygon} is $\P$ equipped with the assignment $\ep$ satisfying $(\star)$. We need to express a lattice multi-polygon as a pair $(\P,\ep)$ to be precise, but we omit $\ep$ and express a lattice multi-polygon simply as $\P$ in the following.  Reduced legal loops introduced in Section~\ref{sect:2} are lattice multi-polygons. 

\begin{rema}
Lattice multi-polygons such that consecutive three points are not on a same line are introduced in \cite[Section 8]{masu99}.  But if we require the condition $(\star)$, then the argument developed there works for any lattice multi-polygon.     
A shaven polygon introduced in \cite{Grum-She-93} is a lattice multi-polygon with $\ep=+$ in our terminology, so that $v_i$ is allowed to lie on the line segment $v_{i-1}v_{i+1}$ but $v_{i-1}$ (resp. $v_{i+1}$) is not allowed to lie on $v_iv_{i+1}$ (resp. $v_{i-1}v_i$) by (2) of $(\star)$, i.e., there is no \emph{whisker}.     
\end{rema}

Let $\P$ be a multi-polygon with a sign assignment $\ep$.  We think of $\P$ as an oriented piecewise linear loop with signs attached to sides.  
For $i=1,\ldots,d$, let $n_i$ denote a normal vector to each side $v_iv_{i+1}$ such that 
the 90 degree rotation of $\ep(v_iv_{i+1})n_i$ has the same direction as $v_iv_{i+1}$. 
The winding number of $\P$ around a point $v \in \R^2 \setminus \P$, 
denoted $d_{\P}(v)$, is a locally constant function on $\R^2 \setminus \P$, 
where $\R^2 \setminus \P$ means the set of elements in $\R^2$ which does not belong to any side of $\P$. 

Following \cite[Section 8]{masu99}, we define 
\[
\begin{split}
A(\P)&:=\int_{v \in \R^2 \setminus \P}d_{\P}(v) dv, \\
B(\P)&:=\sum_{i=1}^d \ep(v_iv_{i+1}) |v_iv_{i+1}|, \\
C(\P)&:=\text{ the rotation number of the sequence of } n_1,\ldots,n_d. 
\end{split}
\]
Notice that $A(\P)$ and $B(\P)$ can be 0 or negative. 
If $\P$ arises from a lattice polygon $P$, namely $\P$ is a sequence of the vertices of $P$ 
arranged in counterclockwise order and $\ep=+$, then $A(\P)=A(P)$,  $B(\P)=B(P)$ and $C(\P)=1$. 

Now, we define $\sharp \P$ in such a way that if $\P$ arises from a lattice polygon $P$, then $\sharp \P=\sharp P$.  Let $\P_+$ be an oriented loop obtained from $\P$ by pushing each side $v_iv_{i+1}$ slightly in the direction of $n_i$. Since $\P$ satisfies the condition $(\star)$, $\P_+$ misses all lattice points, so the winding numbers $d_{\P_+}(u)$ can be defined for any lattice point $u$ using $\P_+$.  Then we define  
\[
\sharp \P := \sum_{u \in \Z^2} d_{\P_+}(u). 
\]

As remarked before, lattice multi-polygons treated in \cite{masu99} are required that 
consecutive three points $v_{i-1}, v_i, v_{i+1}$ do not lie on a same line.  
But if the sign assignment $\ep$ satisfies the condition $(\star)$ above, 
then the argument developed in \cite[Section 8]{masu99} works and 
we obtain the following generalized Pick's formula for lattice multi-polygons as follows.

\begin{theo}[cf.  {\cite[Theorem 8.1]{masu99}}] \label{theo:3.1}
$\sharp \P=A(\P)+\frac{1}{2}B(\P)+C(\P)$.
\end{theo}
\begin{proof}
Let $\P=(v_1,\ldots,v_d)$ be a lattice multi-polygon. 
Similarly to the proof of \cite[Theorem 8.1]{masu99}, 
we construct the multi-fan from $\P$ and apply the results in \cite[Section 7]{masu99}.

Assume that $\P$ contains consecutive three points lying on a line, say, $v_1, v_2$ and $v_3$. 
Let $n_i$ denote the primitive normal vector to each side $v_iv_{i+1}$ such that 
$90$ degree rotation of $\epsilon(v_iv_{i+1})n_i$ has the same direction as $v_iv_{i+1}$. 
Then the condition $(\star)$ implies that $n_1=n_2$. 
Let $n_{12}$ denote the primitive vector such that $n_{12}$ is orthogonal to $n_1$. 
We add the new lattice vector $n_{12}$ between $n_1$ and $n_2$, 
and the remaining method for the construction of multi-fan associated with $\P$ 
is the same as in the proof of \cite[Theorem 8.1]{masu99}. 
Now, by applying the results in \cite[Section 7]{masu99}, we can see that the required formula also holds for $\P$. 
\end{proof}

If we define $\P^\circ$ to be $\P$ with $-\ep$ as a sign assignment, then 
\begin{equation} \label{eq:3.2}
\sharp\P^\circ=A(\P)-\frac{1}{2}B(\P)+C(\P)
\end{equation}
and if $\P$ arises from a lattice polygon $P$, then $\sharp \P^\circ=\sharp P^\circ$.   

Given a positive integer $m$, we dilate $\P$ by $m$ times, denoted $m\P$, in other words, if $\P$ is $(v_1,\dots,v_d)$ with a sign assignment $\ep$, then $m\P$ is $(mv_1,\dots,mv_d)$ with $\ep(v_iv_{i+1})$ as the sign of the side $mv_imv_{i+1}$ of $m\P$ for each $i$. Then we have 
\begin{equation}\label{eq1}
\sharp (m \P) = A(\P) m^2 + \frac{1}{2}B(\P) m + C(\P), 
\end{equation}
that is, $\sharp (m \P)$ is a polynomial in $m$ of degree at most 2 
whose coefficients are as above. Moreover, the equality 
\begin{equation*}\label{eq2}
\sharp (m \P^\circ) = A(\P) m^2 - \frac{1}{2}B(\P) m + C(\P) = (-1)^2 \sharp (-m \P) 
\end{equation*}
holds, so that the reciprocity holds for lattice multi-polygons. 
We call the polynomial \eqref{eq1} the \emph{Ehrhart polynomial} of a lattice multi-polygon $\P$. 
We refer the reader to \cite{BeckRobins} for the introduction to the theory of Ehrhart polynomials of general convex lattice polytopes.

\begin{rema}
In \cite{ha-ma-03}, lattice multi-polytopes $\P$ of dimension $n$ are defined and 
it is proved that $\sharp (m \P)$ is a polynomial in $m$ of degree at most $n$ 
which satisfies $\sharp (m \P^\circ)=(-1)^n \sharp (-m \P)$ 
whose leading coefficient and constant term have similar geometrical meanings to the 2-dimensional case above. 
\end{rema}

\section{Ehrhart polynomials of lattice multi-polygons} \label{sect:4}

In this section, we will discuss which polynomials appear as the Ehrhart polynomials of lattice multi-polygons. 
By virtue of \eqref{eq1}, studying whether a polynomial $am^2+bm+c$ is the Ehrhart polynomial of some lattice multi-polygon 
is equivalent to classifying the triple $(A(\P),\frac{1}{2}B(\P),C(\P))$ for lattice multi-polygons $\P$. 
In the sequel, we will discuss this triple for lattice multi-polygons and their natural subfamilies. 

If the triple $(a,b,c)$ is equal to $(A(\P),\frac{1}{2}B(\P),C(\P))$ of some lattice multi-polygon $\P$, then $(a,b,c)$ must be in the set 
\[
\A=\left\{(a,b,c) \in \frac{1}{2}\Z \times \frac{1}{2}\Z \times \Z : a+b \in \Z \right\}
\] 
because 
\[
B(\P)\in\Z,\quad C(\P)\in \Z,\quad A(\P) + \frac{1}{2}B(\P) + C(\P)=\sharp \P \in \Z.
\]
The following theorem shows that this condition is sufficient. 
  
\begin{theo}\label{mpolygon}
The triple $(a,b,c)$ is equal to $(A(\P),\frac{1}{2}B(\P),C(\P))$ of some lattice multi-polygon $\P$ if and only if $(a,b,c) \in \A$. 
\end{theo}
\begin{proof}
It suffices to prove the \lq\lq if" part.  
We pick up $(a,b,c) \in \A$. Then one has an expression   
\begin{equation} \label{eq:abc}
(a,b,c)=a'(1,0,0)+b' \left(\frac{1}{2},\frac{1}{2},0 \right)+c'(0,0,-1) 
\end{equation}
with integers $a',b',c'$ because $a'=a-b$, $b'=2b$ and $c'=-c$. 
One can easily check that $(1,0,0)$, $(\frac{1}{2},\frac{1}{2},0)$ and $(0,0,-1)$ are respectively equal to $(A(\P_j),\frac{1}{2}B(\P_j),C(\P_j))$ of the lattice multi-polygons $\P_j$ ($j=1,2,3$) shown in Figure~\ref{fig:3}, where the sign of $v_iv_{i+1}$ is given by the sign of $\det(v_i,v_{i+1})$ for $\P_j$. 

\begin{figure}[htb!]
\centering
\includegraphics[scale=0.9]{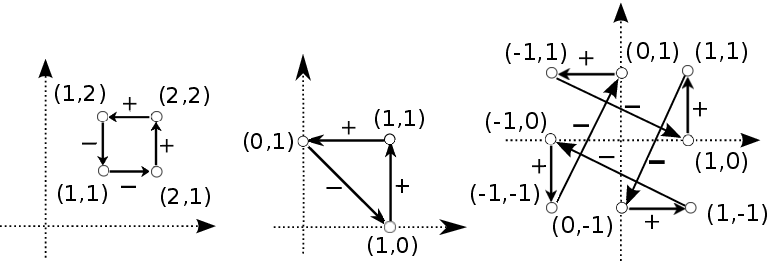}
\caption{lattice multi-polygons $\P_1, \P_2$ and $\P_3$ from the left} \label{fig:3}
\end{figure}

Moreover, reversing both the order of the points and the signs on the sides for $\P_1,\P_2$ and $\P_3$, 
we obtain lattice multi-polygons $\P_1', \P_2'$ and $\P_3'$ whose triples are respectively $(-1,0,0), (-\frac{1}{2}, -\frac{1}{2}, 0)$ and $(0,0,1)$. Since all these six lattice multi-polygons have a common lattice point $(1,1)$, 
one can produce a lattice multi-polygon by joining them as many as we want at the common point and since the triples behave additively with respect to the join operation, this together with \eqref{eq:abc} shows the existence of a lattice multi-polygon with the desired $(a,b,c)$.
\end{proof}

In the rest of the paper, we shall consider several natural subfamilies of lattice multi-polygons and discuss the characterization of their triples.  We note that if $(a,b,c)=(A(\P),\frac{1}{2}B(\P),C(\P))$ for some lattice multi-polygon $\P$, then $(a,b,c)$ must be in the set $\mathcal A$. 

\smallskip

\subsection{Lattice polygons} 
 
One of the most natural subfamilies of lattice multi-polygons would be the family of convex lattice polygons.  Their triples are essentially characterized by P. R. Scott as follows. 

\begin{theo}[\cite{scott76}]\label{scott}
A triple $(a,b,c)\in\mathcal A$ is equal to $(A(P),\frac{1}{2}B(P),C(P))$ of a convex lattice polygon $P$ if and only if $c=1$ and $(a,b)$ satisfies one of the following: \\
\quad {\em (1)} $a+1=b\ge \frac{3}{2}$; \; {\em (2)} $\frac{a}{2}+2\ge b\ge \frac{3}{2}$; \; {\em (3)} $(a,b)=(\frac{9}{2},\frac{9}{2})$. 
\end{theo}

If we do not require the convexity, then the characterization becomes simpler than Theorem~\ref{scott}. 

\begin{prop}\label{prop:notconv}
A triple $(a,b,c)\in\mathcal A$ is equal to $(A(P),\frac{1}{2}B(P),C(P))$ of a (not necessarily convex) lattice polygon $P$ if and only if  $c=1$ and $a+1 \geq b \geq \frac{3}{2}$. 
\end{prop}
\begin{proof}
If $P$ is a lattice polygon, then we have 
\[
C(P)=1,\quad B(P)\ge 3, \quad A(P)-\frac{1}{2}B(P)+1=\sharp P^\circ\ge 0
\]
and this implies the \lq\lq only if" part.  

On the other hand, let $(a,b,1) \in \A$ with $a+1 \geq b \geq \frac{3}{2}$. 
Thanks to Theorem~\ref{scott}, we may assume that $b > \frac{a}{2}+2$, that is, $4b-2a-6 >2$. 
Let $P$ be the lattice polygon shown in Figure~\ref{fig:4}. 
\begin{figure}[htb!] 
\centering
\includegraphics[scale=0.6]{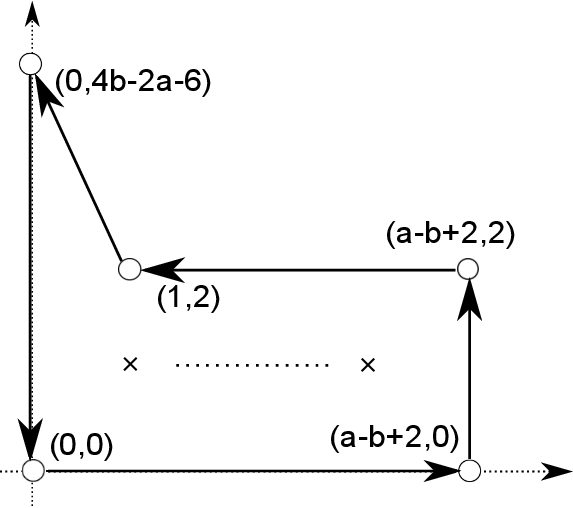}
\caption{a lattice polygon $P$ with $(A(P),\frac{1}{2}B(P),C(P))=(a,b,c)$} \label{fig:4}
\end{figure}
Then, one has 
\[A(P)=2(a-b+2)+\frac{1}{2}(4b-2a-8)=a\] and \[B(P)=(a-b+2)+2+(a-b+1)+1+4b-2a-6=2b.\]
This shows that $(A(P),\frac{1}{2}B(P),C(P))=(a,b,c)$, as desired. 
\end{proof}

\smallskip

\subsection{Unimodular lattice multi-polygons} 

We say that a lattice multi-polygon $\P=(v_1,\ldots,v_d)$ is \emph{unimodular} if the sequence $(v_1,\ldots,v_d)$ is unimodular and the sign assignment $\ep$ is defined by $\ep(v_iv_{i+1})=\det(v_i,v_{i+1})$ for $i=1,\ldots,d$, where $v_{d+1}=v_1$. 
When a unimodular lattice multi-polygon $\P$ arises from a convex lattice polygon, 
$\P$ is essentially the same as so-called a \emph{reflexive polytope} of dimension 2, 
which is completely classified (16 polygons up to equivalence, see, e.g. \cite[Figure 2]{po-rv00}) and the triples $(A(P),\frac{1}{2}B(P),C(P))$ of reflexive polytopes $P$ are characterized by the condition that $c=1$ and $a=b \in \left\{\frac{3}{2},2,\frac{5}{2},3,\frac{7}{2},4,\frac{9}{2} \right\}$. 

We can characterize $(A(P),\frac{1}{2}B(P),C(P))$ of unimodular lattice multi-polygons $P$ as follows. 

\begin{theo}\label{umpolygon}
A triple $(a,b,c)\in\mathcal A$ is equal to $(A(\P),\frac{1}{2}B(\P),C(\P))$ of a unimodular lattice multi-polygon $\P$ if and only if $a=b$.  
\end{theo}

\begin{proof}
If $\P$ is a unimodular lattice multi-polygon arising from a unimodular sequence $v_1,\dots,v_d$, then one sees that 
\[
\begin{split}
A(\P)&=\frac{1}{2}\sum_{i=1}^d\det(v_i,v_{i+1})\\
B(\P)&=\sum_{i=1}^d\det(v_i,v_{i+1})|v_iv_{i+1}|=\sum_{i=1}^d\det(v_i,v_{i+1})
\end{split}
\]
and this implies the \lq\lq only if" part.

Conversely, if $(a,b,c) \in \A$ satisfies $a=b$, then one has an expression  
\[
(a,b,c)=a' \left(\frac{1}{2},\frac{1}{2},0 \right)+c'(0,0,-1)
\]
with integers $a',c'$ because $a'=2a$ and $c'=-c$.  We note that the lattice multi-polygons $\P_2, \P_3, \P_2'$ and $\P_3'$ in the proof of Theorem~\ref{mpolygon} are unimodular lattice multi-polygons.  Therefore, joining them as many as we want at the common point $(1,1)$, we can find a unimodular lattice multi-polygon $(A(P),\frac{1}{2}B(P),C(P))=(a,b,c)$, as required. 
\end{proof}

\begin{exam}
The $\P$ and $\Q$ in Example~\ref{ex1} are unimodular lattice multi-polygons and we have 
$$\left(A(\P),\frac{1}{2}B(\P),C(\P)\right)=\left(\frac{3}{2},\frac{3}{2},1\right) \text{ and } \left(A(\Q),\frac{1}{2}B(\Q),C(\Q)\right)=(3,3,2).$$ 
\end{exam}

\smallskip

\subsection{Some other subfamilies of lattice multi-polygons}

\begin{exam}[Left-turning (right-turning) lattice multi-polygons]\label{lrpolygon}
We say that a lattice multi-polygon $\P$ is \emph{left-turning} (resp. \emph{right-turning}) 
if $\det(v-u,w-u)$ is always positive (resp. negative) for consecutive three points $u,v,w$ in $\P$ arranged in this order 
 not lying on a same line. In other words, $w$ lies in the left-hand side (resp. right-hand side) with respect to the direction from $u$ to $v$. For example, $\P_1$, $\P_2$ and $\P_3$ in Figure~\ref{fig:3} and $\Q$ in Example~\ref{ex1} (b) are all left-turning. 

Somewhat suprisingly, the left-turning (or right-turning) condition does not give any restriction on the triple $(A(\P),\frac{1}{2}B(\P),C(\P))$, 
that is, every $(a,b,c)\in\A$ can be equal to $(A(\P),\frac{1}{2}B(\P),C(\P))$ of a left-turning (or right-turning) lattice multi-polygon $\P$.  
A proof is given by using the lattice multi-polygons $\P_1,\P_2, \P_3$ shown in Figure \ref{fig:3} together with $\P_4,\P_5, \P_6$ shown in Figure \ref{fig:5}. 
Remark that the signs of $\P_4,\P_5$ and $\P_6$ do not always coincide with the sign of $\det(v_i,v_{i+1})$.

\begin{figure}[htb!]
\centering
\includegraphics[scale=0.63]{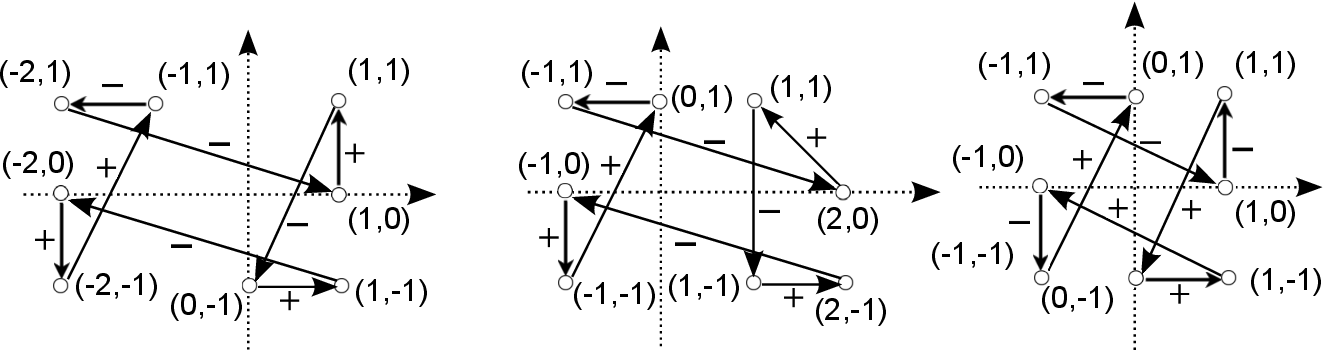}
\caption{lattice multi-polygons $\P_4, \P_5$ and $\P_6$ from the left} \label{fig:5}
\end{figure}
\end{exam}

\smallskip

\begin{exam}[Left-turning lattice multi-polygons with all $+$ signs]
We consider left-turning lattice multi-polygons $\P$ and impose one more restriction that the signs on the sides of $\P$ are all $+$. In this case, some interesting phenomena happen. For example, a simple observation shows that 
\begin{equation}\label{leftcoil}
B(\P) \geq 2 C(\P)+1 \;\;\text{and}\;\; C(\P) \geq 1. 
\end{equation}
We note that $C(\P)=1$ if and only if $\P$ arises from a convex lattice polygon, and those $(A(\P),\frac{1}{2}B(\P),C(\P))$ are characterized by Theorem~\ref{scott}.  Therefore, it suffices to treat the case where $C(\P)\ge 2$ 
and we can see that 
a triple $(a,b,c)\in\A$ is equal to $(A(\P),\frac{1}{2}B(\P),C(\P))$ of a left-turning lattice multi-polygon $\P$ with all $+$ signs if 
\[b \geq c + 1 \;\text{ and }\; c \geq 2.\] 
This condition is equivalent to $B(\P)\ge 2C(\P)+2$ for a lattice multi-polygon. On the other hand, we have $B(\P)\ge 2C(\P)+1$ for a left-turning lattice multi-polygon $\P$ with all $+$ signs by \eqref{leftcoil}.  Therefore, the case where $B(\P)=2C(\P)+1$ is not covered above and this extreme case is exceptional. 
In fact, one can observe that if $\P$ is a left-turning multi-polygon with all $+$ signs and $B(\P)=2C(\P)+1$, then $\sharp \P^\circ \geq 0$, that is, $A(\P) \geq \frac{1}{2}$. 
\end{exam}

\smallskip

\begin{exam}[Lattice multi-polygons with all $+$ signs]
Finally, we consider lattice multi-polygons $\P$ with all $+$ signs, namely, we do not assume that 
$\P$ is either left-turning or right-turning. 
However, this case is similar to the previous one (left-turning lattice multi-polygons with all $+$ signs). 
For example, when $C(\P) \not= 0$, we still have $B(\P) \geq 2 |C(\P)|+1$. Thus, we also have that 
a triple $(a,b,c) \in \mathcal A$ is equal to $(A(\P),\frac{1}{2}B(\P),C(\P))$ of a lattice multi-polygon $\P$ with all $+$ signs 
if \[b \geq |c|+1 \;\text{ and } |c| \geq 2.\] 
Moreover, when $B(\P)=2|C(\P)|+1$, $\P$ must be left-turning or right-turning according as $C(\P)>0$ or $C(\P)<0$. Hence, we can say that 
when we discuss $(A(\P),\frac{1}{2}B(\P),C(\P))$ of lattice multi-polygons $\P$ with all $+$ signs, 
it suffices to consider those of left-turning or right-turning ones when $C(\P) \not\in \{-1,0,1\}$. 

On the other hand, on the remaining exceptional cases where $C(\P)=0$ or $C(\P)=\pm 1$, 
we can characterize the triples completely as follows. Let $(a,b,c) \in \mathcal A$. 
\begin{itemize}
\item[(a)] When $c=0$, 
$(a,b,c)$ is equal to $(A(\P),\frac{1}{2}B(\P),C(\P))$ of a lattice multi-polygon $\P$ with all $+$ signs 
if and only if $b \geq 2$. See Figure \ref{fig:c=0}. 
\item[(b)] When $c=1$, 
$(a,b,c)$ is equal to $(A(\P),\frac{1}{2}B(\P),C(\P))$ of a lattice multi-polygon $\P$ with all $+$ signs 
if and only if either $b \geq \frac{5}{2}$ or $\frac{3}{2} \leq b \leq 2$ and $a-b+1 \geq 0$. See Figure \ref{fig:c=1} and Proposition \ref{prop:notconv}. 
\item[(c)] When $c=-1$, 
$(a,b,c)$ is equal to $(A(\P),\frac{1}{2}B(\P),C(\P))$ of a lattice multi-polygon $\P$ with all $+$ signs 
if and only if either $b \geq \frac{5}{2}$ or $\frac{3}{2} \leq b \leq 2$ and $a+b-1 \leq 0$. 
One can simply reverse the order of the vertices and flip the sign of $a$ of the example in Figure \ref{fig:c=1}.
\end{itemize}

\begin{figure}[htb!]
\centering
\includegraphics[scale=0.6]{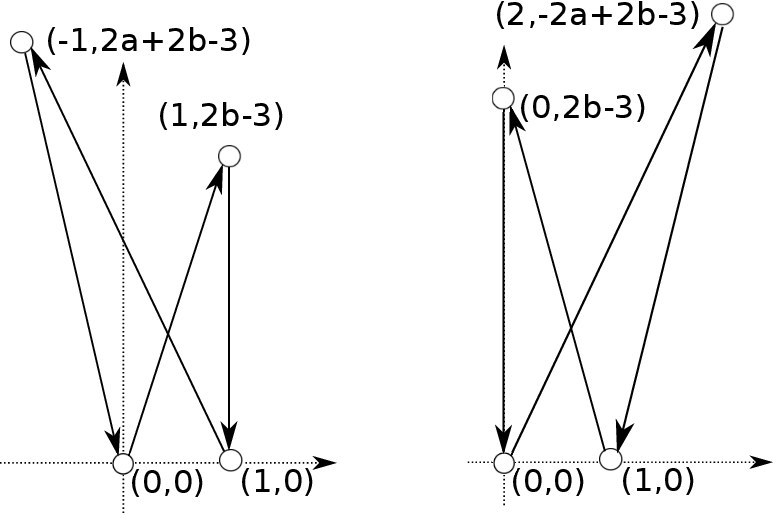}
\caption{lattice multi-polygons with all $+$ signs whose triples equal $(a,b,0)$ when $a+b \geq 2$ and $a+b \leq 2$, respectively}\label{fig:c=0}
\end{figure}

\begin{figure}[htb!]
\centering
\includegraphics[scale=0.6]{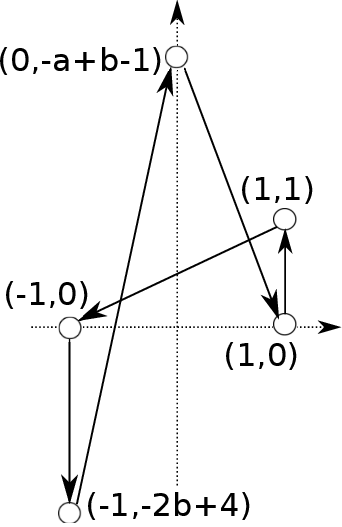}
\caption{a lattice multi-polygon with all $+$ signs whose triple equals $(a,b,1)$ when $b \geq \frac{5}{2}$} \label{fig:c=1}
\end{figure}
\end{exam}

\newpage

\appendix
\section{Proof of Theorem \ref{rotation} using toric topology}

{Theorem~\ref{rotation} was originally proved using toric topology.  In fact, it is} proved in \cite[Section 5]{masu99} when $\e_i=1$ for every $i$ and the argument there works in our general setting with a little modification, which we shall explain.  

We identify $\Z^2$ with $H_2(BT)$ where $T=(S^1)^2$ and $BT$ is the classifying space of $T$. We may think of $BT$ as $(\C P^\infty)^2$.  For each $i$ $(i=1,\dots, d)$, we form a cone $\angle v_iv_{i+1}$ in $\R^2$ spanned by $v_i$ and $v_{i+1}$ and attach the sign $\e_i$ to the cone.  The collection of the cones $\angle v_iv_{i+1}$ with the signs $\e_i$ attached form a multi-fan $a_j$ and the same construction as in \cite[Section 5]{masu99} produces a real $4$-dimensional closed connected smooth manifold $M$ with an action of $T$ satisfying the following conditions:
\begin{enumerate}
\item $H^{odd}(M)=0$.
\item $M$ admits a unitary (or weakly complex) structure preserved under the $T$-action and the multi-fan associated to $M$ with this unitary structure is the given $a_j$.
\item Let $M_i$ $(i=1,\dots,d)$ be the characteristic submanifold of $M$ corresponding to the edge vector $v_i$, that is, $M_i$ is a real codimension two submanifold of $M$ fixed pointwise under the circle subgroup determined by the $v_i$.   Then $M_i$ does not intersect with $M_j$ unless $j=i-1,i,i+1$ and the intersection numbers of $M_i$ with $M_{i-1}$ and $M_{i+1}$ are $\e_{i-1}$ and $\e_i$ respectively.
\end{enumerate}

Choose an arbitrary element $v\in \R^2$ not contained in any one-dimensional cone in the multi-fan $a_j$.  Then Theorem 4.2 in \cite{masu99} says that the Todd genus $T[M]$ of $M$ is given by 
\begin{equation} \label{todd}
T[M]=\sum_i \e_i, 
\end{equation}
where the sum above runs over all $i$'s such that the cone $\angle v_iv_{i+1}$ contains the vector $v$.  Clearly the right hand side in \eqref{todd} agrees with the rotation number of the sequence $v_1,\dots,v_d$ around the origin.  In the sequel, we compute the Todd genus $T[M]$.  

Let $ET\to BT$ be the universal principal $T$-bundle and $M_T$ the quotient of $ET\times M$ by the diagonal $T$-action.  The space $M_T$ is called the Borel construction of $M$ and the equivariant cohomology $H^q_T(M)$ of the $T$-space $M$ is defined to be $H^q(M_T)$.  The first projection from $ET\times M$ onto $ET$ induces a fibration 
\[
\pi\colon M_T\to ET/T=BT
\] 
with fiber $M$.  The inclusion map $\iota$ of the fiber $M$ to $M_T$ induces a surjective homomorphism $\iota^*\colon H_T^q(M)\to H^q(M)$.  

Let $\xi_i\in H^2_T(M)$ be the Poincar\'e dual to the cycle $M_i$ in the equivariant cohomology.  The $\xi_i$ restricts to the ordinary Poincar\'e dual $x_i\in H^2(M)$ to the cycle $M_i$ through the $\iota^*$.  By Lemma 1.5 in \cite{masu99}, we have 
\begin{equation} \label{1.5}
\pi^*(u)=\sum_{j=1}^d\langle u,v_j\rangle \xi_j \quad \text{for any $u\in H^2(BT)$}, 
\end{equation}
where $\langle\ ,\ \rangle$ denotes the natural pairing between cohomology and homology.  Multiplying the both sides of \eqref{1.5} by $\xi_i$ and restricting the resulting identity to the ordinary cohomology by $\iota^*$, we obtain  
\begin{equation} \label{uvx}
0=\langle u, v_{i-1}\rangle x_{i-1}x_i+\langle u,v_i\rangle x_i^2 +\langle u,v_{i+1}\rangle x_{i+1}x_i \quad \text{for all $u\in H^2(BT)$}
\end{equation}
because $M_i$ does not intersect with $M_j$ unless $j=i-1,i,i+1$, where $x_{d+1}=x_1$.  We evaluate the both sides of \eqref{uvx} on the fundamental class $[M]$ of $M$.  Since the intersection numbers of $M_i$ with $M_{i-1}$ and $M_{i+1}$ are respectively $\e_{i-1}$ and $\e_i$ as mentioned above, the identity \eqref{uvx} reduces to   
\begin{equation} \label{u}
0=\langle u, v_{i-1}\rangle \e_{i-1}+ \langle u,v_i\rangle \langle x_i^2,[M]\rangle+\langle u,v_{i+1}\rangle \e_i \quad \text{for all $u\in H^2(BT)$}
\end{equation}
and further reduces to 
\begin{equation} \label{xi2}
 0=\e_{i-1}v_{i-1}+\langle x_i^2,[M]\rangle v_i +\e_i v_{i+1}
\end{equation}
because \eqref{u} holds for any $u\in H^2(BT)$.  
Comparing \eqref{xi2} with \eqref{ai}, we conclude that $\langle x_i^2,[M]\rangle=a_i$.  Summing up the above argument, we have 
\begin{equation} \label{sumup}
\langle x_ix_j,[M]\rangle =\begin{cases} \e_{i-1} \quad&\text{if $j=i-1$},\\
a_i \quad&\text{if $j=i$},\\
\e_i \quad&\text{if $j=i+1$},\\
0 \quad&\text{otherwise}.
\end{cases}
\end{equation}

By Theorem 3.1 in \cite{masu99} the total Chern class $c(M)$ of $M$ with the unitary structure is given by $\prod_{i=1}^d(1+x_i)$.  Therefore 
\[
c_1(M)=\sum_{i=1}^dx_i,\quad\quad c_2(M)=\sum_{i<j}x_ix_j
\]
and hence 
\[
\begin{split}
T[M]&=\frac{1}{12}\langle c_1(M)^2+c_2(M),[M]\rangle \\
&=\frac{1}{12}\langle (\sum_{i=1}^dx_i)^2+\sum_{i<j}x_ix_j,[M]\rangle\\
&=\frac{1}{12}(\sum_{i=1}^d a_i+3\sum_{i=1}^d\e_i), 
\end{split}
\]
where the first identity is known as Noether's formula when $M$ is an algebraic surface and known to hold even for unitary manifolds, and we used \eqref{sumup} at the last identity.  This proves the theorem because $T[M]$ agrees with the desired rotation number as remarked at \eqref{todd}.


\begin{thebibliography}{19}

\bibitem{BeckRobins}
M. Beck and S. Robins,
``Computing the Continuous Discretely'', 
Undergraduate Texts in Mathematics,
Springer, 2007.

\bibitem{cast12}
W. Castryck, 
\emph{Moving out the edges of a lattice polygon}, 
Discrete Comput. Geom. 47 (2012), 496--518. 

\bibitem{ce-re-sk05}
M. Cencelj, D. Repovs, and M. Skopenkov, 
\emph{A short proof of the twelve lattice point theorem}, 
Math. Notes 77: no. 1-2 (2005), 108--111. 

\bibitem{diro95}
R.~Diaz and S.~Robins, 
\emph{Pick's Formula via the Weierstrass $\wp$-Function}, 
Amer. Math. Monthly, 102 (1995), 431--437.

\bibitem{fult93}
W. Fulton,
\emph{An introduction to toric varieties},
Ann. of Math. Studies, vol.~113, Princeton Univ. Press, Princeton, N.J.,
1993.

\bibitem{Grum-She-93}
B. Gr\"{u}nbaum and G. C. Shephard, 
\emph{Pick's theorem}, 
Amer. Math. Monthly 100 (1993), 150--161. 

\bibitem{ha-ma-03}
A. Hattori and M. Masuda, 
\emph{Theory of multi-fans}, 
Osaka J. Math. 40 (2003), 1--68. 

\bibitem{ka-ni11}
A. M. Kasprzyk and B. Nill,
\emph{Reflexive polytopes of higher index and the number 12},
arXiv:1107.4945. 

\bibitem{masu99}
M. Masuda,
\emph{Unitary toric manifolds, multi-fans and equivariant index},
Tohoku Math. J. 51 (1999), 237--265.

\bibitem{oda88}
T. Oda,
\emph{Convex bodies and algebraic geometry} (An Introduction to the 
theory of toric varieties), Ergebnisse der Mathematik und ihrer Grenzgebiete, 
3. Folge Band 15, A Series of Modern Surveys in Mathematics, Springer-Verlag, 1988. 

\bibitem{pick99}
G. Pick, \emph{Geometrisches zur Zahlentheorie}, 
Sitzenber. Lotos (Prague) 19 (1899), 311-319. 

\bibitem{po-rv00}
B. Poonen and F. Rodriguez-Villegas, 
\emph{Lattice Polygons and the Number 12}, 
Amer. Math. Monthly 107 (2000), pp. 238 -- 250.


\bibitem{scott76}
P. R. Scott, On convex lattice polygons,
{\em Bull. Austral. Math. Soc.} {\bf 15} (1976), 395 -- 399.

\bibitem{ziv12}
R. T. Zivaljevic, 
\emph{Rotation number of a unimodular cycle: an elementary approach}, 
{\em Discrete Math.} {\bf 313} (2013), no. 20, 2253 -- 2261. 

\end{thebibliography}
\end{document}